%% file: bgV1.tex
\documentclass[12pt,a4paper]{amsart}
\usepackage[T1]{fontenc}
\usepackage{amssymb}
\usepackage{amsmath}
\usepackage{amsfonts}
\usepackage{ifthen}

\usepackage[all]{xy}

\newcommand{\figs}{\usepackage{graphicx}
\usepackage{color} \graphicspath{{figs/}} \numberwithin{figure}{section} }

\figs

\newcounter{margnotes}

\def\sideremark#1{\ifvmode\leavevmode\fi\vadjust{\vbox to0pt{\vss 
      \hbox to 0pt{\hskip\hsize\hskip1em           
 \vbox{\hsize3cm\tiny\raggedright\pretolerance10000
 \noindent #1\hfill}\hss}\vbox to8pt{\vfil}\vss}}}%

\newcounter{lemenumi}
\newcommand{\labelemenumi}{(\alph{lemenumi})}

\newtheorem{definition}{Definition}

\newtheorem{theorem}{Theorem} 
\newtheorem{lemma}{Lemma}
\newtheorem{proposition}{Proposition}
\newtheorem{corollary}{Corollary}
\newcommand{\bbC}{\mathbb{C}}

\newcommand{\bbZ}{\mathbb{Z}}

\newcommand{\ep}{\varepsilon}

\newcommand{\idd}{\mathrm{Id}}

\newcommand{\spn}{\mathrm{Span}}

\DeclareMathOperator{\mn}{\mathcal{M}}
\DeclareMathOperator{\mon}{\mathcal{M}}

\title{On the reduction of the degree of linear differential operators}
\author{Marcin Bobie\'nski}
\author{Lubomir Gavrilov}
\thanks{Supported by Polish MNiSzW Grant No N N201 397937}
\date{\today}

\begin{document}
\begin{abstract}
 Let $L$ be a linear differential operator with coefficients in some differential field $k$
of characteristic zero with algebraically closed field of constants. Let  $k^a$
be the  algebraic closure of $k$. For a solution $y_0$, $Ly_0=0$, we determine
the linear differential operator of minimal degree $\widetilde{L}$ and
coefficients in $k^a $, such that $\widetilde{L}y_0=0$. This result is then
applied to some Picard-Fuchs equations which appear in the study of
perturbations of plane polynomial vector fields of Lotka-Volterra type.
\end{abstract}
\maketitle
\section{Introduction}
\label{sec:intro} Let $y_0=y_0(t)$ be a solution of the linear differential
equation
 \begin{equation}\label{aa}
a_0(t) y^{(n)} + a_1(t) I^{(n-1)} + \dots a_n(t) I = 0
\end{equation}
 where $a_i\in
k=\mathbb{C}(t)$ are functions, rational in the independent variable $t$. We
are interested in determining the equation of minimal degree $d\leq n$
\begin{equation}\label{bb}
b_0(t) y^{(d)} + b_1(t) y^{(d-1)} + \dots b_d(t) y = 0
\end{equation}
such that
\begin{itemize}
    \item $y_0$ is a solution
    \item the coefficients $b_i$ are algebraic functions in $t$
\end{itemize}
Recall that a function $b(t)$ is said to be algebraic in $t$ if there exists a
polynomial $P$ with coefficients in $k=\mathbb{C}(t)$, such that $P(b(t))\equiv
0$.

We shall suppose that, more generally, $k$ is an arbitrary differential field
of characteristic zero with algebraically closed field of constants, $k^a$ is
its algebraic closure, and $a_i\in k$, $b_j\in k^a$. To find the equation
$(\ref{bb})$ we consider the differential Galois group $G$ of $(\ref{aa})$ and
its connected subgroup $G^0$ containing the unit element of $G$. Our first
result, Theorem \ref{main}, says that the orbit of $y_0$ under the action of
$G^0$ spans the solution space of $(\ref{bb})$.

A particular attention is given further to the case in which (\ref{aa}) is of
Fuchs or Picard-Fuchs type. Theorem \ref{main} is re-formulated in terms of the
action of the corresponding monodromy groups in Theorem \ref{th:main1} and
Theorem \ref{th:algeq}.

In the last part of the paper, section \ref{sec:ex}, we apply the general
theory to some Abelian integrals appearing in the study of perturbations of the
Lotka-Volterra system. These integrals have the form
$$
I(t) = \int_{\gamma(t)} \omega
$$
where
$$\gamma(t)\subset \{ (x,y)\in\mathbb{C}^2 : F(x,y)=t\}$$
is a continuous family of ovals,
$$
F(x,y)= x^py^p(1-x-y) \mbox{  or  } F(x,y)= x^p (y^2-x-1)^q,\;\; p,q\in
\mathbb{N}
$$
and $\omega$ is a suitable rational one-form on $\mathbb{C}^2$. In the first
case the Abelian integral satisfies a Picard-Fuchs equation of order $2p+2$. It
has been shown by van Gils and Horozov \cite{hvg}, that $I(t)$ satisfies also a
second order differential equation whose coefficients are functions algebraic
in $t$. This allows to compute the zeros of $I(t)$ (by the usual Roll's theorem
for differential equations) and finally, to estimate the number of limit cycles
of the perturbed plane foliation defined by
$$
dF +\varepsilon \widetilde{\omega}=0
$$
where $\widetilde{\omega}$ is a real polynomial one-form on $\mathbb{R}^2$. By
making use of  Theorem \ref{main} we provide  the theoretical explanation of
the phenomenon observed observed first in \cite{hvg}, see section
\ref{sec:hvg}. Another interesting case, studied in the paper is when $F(x,y)=
x^p (y^2-x-1)^q$ ($p,q$ relatively prime). The Abelian integral $I(t)$
satisfies a Picard-Fuchs equation of order $p+q+1$, which is the dimension of
the first homology group of the generic fiber $F^{-1}(t)$. We show that the
minimal order of the equation (\ref{bb}) is $p+q+1$ or $p+q$ or $p+q-1$, and
that the coefficients $b_i(t)$ are rational in $t$, see section
\ref{sec:parab}. The meaning of this is that the differential Galois group of
the Picard-Fuchs equation is connected and, in contrast to \cite{hvg}, there is
no reduction of the degree, which may only drop by one or two, depending on
whether $\omega$ has or has not residues "at infinity".

\begin{center}
\emph{Acknowledgements}\end{center}

 Part of the paper was written while the first author was
visiting the University Paul Sabatier of Toulouse. He is obliged for the
hospitality.

\section{Statement of the result}
\label{sec:gen}

Let $k$ be a differential field of characteristic zero with algebraically
closed field of constants $C$, $E\supset k$ be a Picard-Vessiot extension for
the homogeneous monic linear differential operator $L$
\begin{equation}\label{l}
L(y)= y^{(n)} + a_{n-1}y^{(n-1)}+ \dots + a_1 y^{(1)} + a_0 y, a_i \in k
\end{equation}
and $y_0\in E$ a solution, $L(y_0)=0$. We denote by $k^a \supset k$ the
algebraic closure of $k$ which is also a differential field.
\begin{definition}
\emph{A homogeneous monic linear differential operator $\widetilde{L}$ with
coefficients in $k^a$ is said to be annihilator of $y_0$, provided that
$\widetilde{L}(y_0)=0$. The annihilator $\widetilde{L}$ is said to be minimal,
provided that its degree is minimal.}
\end{definition}
 The
definition has a sense, because the algebraic closure $E^a$ of $E$ is a
differential field which contains $E$ and $k^a$ as differential subfields. The
minimal annihilator obviously exists and is unique, its degree is bounded by
the degree of $L$ which is an annihilator of $y_0$.

We are interesting in the following question

\emph{For a given solution $y_0$ as above, find the corresponding minimal
annihilator $\widetilde{L}$}

To answer, consider the differential Galois group $G=Gal(E/k)$, which is the
group of differential automorphisms of $E$ fixing $k$. Recall that $G$ is an
algebraic group over $C$, and let
 $G^0$ be the connected component of $G$, containing the unit element (the identity). The intermediate
 field $\widetilde{k}=E^{G^0}$, $k\subset \widetilde{k} \subset E$, of elements
invariant under $G^0$ is then a finite algebraic extension of $k$. We denote it
by $\widetilde{k}$.

Let $y_0, y_1, \dots , y_{d-1}$ be a basis of the $C$-vector space spanned by
the orbit
 $$
G^0y_0 = \{g(y_0): g\in G^0\}\subset E
 $$
and consider the Wronskian determinant in $s$ variables
$$
W(y_1,y_2,\dots,y_s) = \det \left(%
\begin{array}{cccc}
  y_1 & y_2 & \dots & y_s \\
  y_1' & y_2' & \dots & y_s' \\
  \vdots & \vdots & \ddots & \vdots \\
 y_1^{(s-1)} & y_2^{(s-1)}  & \ldots & y_{s-1}^{(s-1)}  \\
\end{array}%
\right) .
$$
$y_0$ satisfies the differential equation
\begin{equation}
W(y,y_0, y_1, \dots , y_{d-1})=0
\end{equation}
and because of the $C$-linear independence of $y_i$
$$
W(y_0,y_1,\dots,y_{d-1})\neq 0 .
$$
Let $\widetilde{L}$ be the monic linear differential operator defined by
\begin{equation}\label{minimal}
\widetilde{L} (y) = \frac{W(y,y_0,y_1,\dots,y_{d-1})}{W(y_0,y_1,\dots,y_{d-1})}
.
\end{equation}
Its coefficients are invariant under the action of $G^0$, and hence they belong
to the differential field $ \widetilde{k}=E^{G^0} . $

Our first result is the following
\begin{theorem}
\label{main} The differential operator $\widetilde{L}$ (\ref{minimal}) is the
minimal annihilator of the solution $y_0$.
\end{theorem}
\begin{proof}
Let $L_{min}$ be the unique differential operator of minimal degree with
coefficients in some algebraic extension $k_{min}$ of $k$, such that
$L_{min}(y_0)=0$. Denote by $E_{min}$ the Picard-Vessiot extension for
$L_{min}$.

As a first step, we shall show that $E_{min}$ can be identified to a
differential subfield of the Picar-Vessiot extension $E$ for $L$. The algebraic
closure $E^a$ of $E$ is a differential field which contains $E$ and every
algebrtaic extension of $k$ (hence it contains  $k_{min}$). Therefore the
compositum $k_{min} E$ of $ k_{min}$ and $  E$, that is to say the smallest
field containing $E$ and $\tilde{k}$, is well defined \cite{lang}. The
differential automorphisms group $Gal(k_{min} E/k)$ acts on the compositum
$k_{min} E$ and leaves $E$ invariant. Therefore $Gal(k_{min} E/k_{min})\subset
Gal(k_{min} E/k)$ leaves $E$ invariant too, and the orbit $Gal(k_{min}
E/k_{min}) y_0$ is contained in $E$. Let $y_0, y_1, \dots , y_{m-1}$ be a basis
of the $C$-vector space spanned by this orbit. Then $y_0$ satisfies the
differential equation
\begin{equation}\label{wronskian}
\frac{W(y,y_0, y_1, \dots , y_{m-1})}{W(y_0, y_1, \dots , y_{m-1})}=0
\end{equation}
and the coefficients of the corresponding monic linear homogeneous differential
operator belong to $k_{min}$.

Consider the ring of differential polynomials
$$
k_{min}\{Y\}= k_{min}[{Y^{(i)}: i=0,1,2,\dots }]
$$
in formal variables $Y^{(i)}$. Identifying differential operators on
$\widetilde{k}$ to polynomials ( the derivatives $y^{(i)}$ correspond to
variables $Y^{(i)}$ ), we may consider the ideal $I$ generated by homogeneous
linear differential operators with coefficients in $k_{min}$ which annihilate
$y_0$. This is obviously a linear ideal which, according to the general theory
(see \cite[Proposition 1.8]{magid} ), is principal in the following sense.
There exists a linear differential operator  with coefficients in $k_{min}$,
which generates $I$. Clearly the generator of $I$ is the operator $L_{min}$
defined above. It follows that the solution space of $L_{min}$ can be
identified to a $C$-vector subspace of the solution space of the operator
defined by (\ref{wronskian}), which implies
\begin{equation}
\label{kkee} k \subset k_{min} \subset E_{min} \subset E .
\end{equation}
(the first two inclusions hold by definition).

At the second step of the proof we shall show that $\deg L_{min} = \deg
\widetilde{L}$. Indeed, the automorphisms group $G^0$ leaves fixed the elements
of $E$ which are algebraic on $k$. In particular, the elements of $k_{min}$ are
fixed by $G^0$ and hence $G^0$ induces differential automorphisms of the
Picard-Vessiot extension $E_{min}$. This shows that the solution space of
$L_{min}$ contains the solution space of $\widetilde{L}$ and
$$
\deg L_{min} \geq \deg \widetilde{L} .
$$

Reciprocally, if we consider (by the construction above) the ideal in
$\widetilde{k}\{Y\}$ generated by all linear homogeneous differential operators
with coefficients in $\widetilde{k}$, which annihilate $y_0$, then this ideal
is linear and principal. The generator of the ideal corresponds to the operator
$L_{min}$, and hence
$$
\deg L_{min} \leq \deg \widetilde{L} .
$$
Theorem \ref{main} is proved.
\end{proof}

To the end of this section we apply Theorem \ref{main} to Fuchs and
Picard-Fuchs differential operators. The minimal annihilator of a solution is
described in terms of the action of the monodromy group.

Let $L$  be a Fuchsian differential operator of order $n$ on the Riemann sphere
$\mathbb{P}^1$,  $\Delta=\{ t_1,\ldots,t_s,\infty \}$ be the set of its
singular points. The field of constants is $C=\mathbb{C}$, the coefficients of
$L$ belong to the field of rational functions $k=\bbC(t)$. Denote by $S\cong
\mathbb{C}^n$ the complex vector space of solutions of $L=0$. The monodromy
group $\mn$ of $L$ is the image of the homomorphism (monodromy representation)
$$
\pi_1(\mathbb{P}^1 \setminus \Delta,*) \rightarrow GL(S) .
$$
The Zariski closure of $\mn$ in $GL(S)$ is the differential Galois group $G$ of
$L$:
$$\overline {\mn}=G .$$
A vector subspace $V \subset S$ is invariant under  the action of $G$ if and
only if it is invariant under the action of $\mn$. A subspace $V \subset S$ is
said to be \emph{virtually invariant}, provided that it is invariant under the
action of identity component $G^0$ of $G$, or equivalently, under the action of
$\mn\bigcap G^0$. For an automorphism $g\in G$ the set $g(V)\subset S$ is a
vector subspace of the same dimension. Thus $G$ acts on the Grassmannian space
$Gr(d,S)$
$$
G \times Gr(d,S) \rightarrow Gr(d,S): (g,V) \mapsto g(V) .
$$
and for every
plane $V\in Gr(d,S)$ the orbit
$$
G(V)= \{ g(V): g\in G\}\subset Gr(d,S)
$$
is well defined.
\begin{lemma}
\label{virtual} A plane $V\in Gr(d,S)$ is virtually invariant, if and only if
the orbit $G(V)\subset Gr(d,S)$ is finite.
\end{lemma}
\begin{proof}
We have
$$
\overline{\mn(V)} = \overline{\mn}(V) = G(V) \supset G^{0} (V)
.
$$
If the orbit $\mn(V)$ is finite, then $\mn(V)=
\overline{\mn(V)}$ and hence $G^{0}(V)$ is finite. As $G^0$ is a
connected Lie group, then $G^{0}(V)= V$ and $V$ is virtually invariant.

Suppose that $V$ is virtually invariant. As $G/G^0$ is a finite group, then
$G^{0}(V)= V$ implies that the orbit $\overline{\mn}(V) = G(V)\subset
Gr(d,S)$ is finite and hence $\mn(V) \subset  \overline{\mn}(V)$
is finite too.\\
\end{proof}
Let $L$ be a Fuchsian differential operator as above, and $y_0$ a solution,
$L(y_0)=0$. The minimal annihilator of $y_0$ is a differential operator
$\widetilde{L}$ of minimal degree with coefficients in some algebraic extension
of $\mathbb{C}(t)$. Thus $\widetilde{L}$ is a Fuchsian operator too, but on a
suitable compact Riemann surface realized as a finite covering of
$\mathbb{P}^1$. Let $V_1, V_2\subset S$ be two virtually invariant planes
containing the solution $y_0$. Then $V_1\cap V_2$ is a virtually invariant
plane containing $y_0$. This shows the existence of a unique virtually
invariant plane $V$ of minimal dimension, containing $y_0$. We call such a
plane minimal. According to Lemma \ref{virtual} and Theorem \ref{main} the
minimal annihilator of $y_0$ is constructed as follows. Let
$y_0,y_1,\dots,y_{d-1}$ be a basis of the minimal virtually invariant plane $V$
containing $y_0$. Consider the Fuchsian differential operator defined as in
formula (\ref{minimal}).
\begin{theorem}
\label{th:main1} The differential operator $\widetilde{L}$ is the minimal
annihilator of the solution $y_0$. The degree of $\widetilde{L}$ equals the
dimension of the minimal virtually invariant plane containing $y_0$.
\end{theorem}

Suppose finally that $L$ is a linear differential operator of Picard-Fuchs (and
hence of Fuchs) type. We shall adapt Theorem \ref{th:main1} to this particular
setting.

Let $F: \bbC^2 \rightarrow \bbC$ be a bivariate non-constant polynomial. It is
known that there is a finite number of atypical points $\Delta=\{
t_1,\ldots,t_n \}$, such that the fibration defined by $F$

\begin{equation}\label{milnor}
F:\bbC^2\setminus F^{-1}(\Delta) \rightarrow \bbC\setminus \Delta
\end{equation}
is locally trivial. The fibers $F^{-1}(t)$, $t\not\in \Delta$ are open Riemann
surfaces, homotopy equivalent to a bouquet of a finite number of circles.
Consider also the associated homology and co-homology bundles with fibers
$H_1(F^{-1}(t),\mathbb{C})$ and $H^1(F^{-1}(t),\mathbb{C})$ respectively. Both
of these vector bundles carry a canonical flat connection. Choose a locally
constant section $\gamma(t)\in H_1(F^{-1}(t),\mathbb{C}$ and consider the
Abelian integral
\begin{equation}\label{abelian}
I(t)= \int_{\gamma(t)}\omega
\end{equation}
where $\omega$ is a meromorphic one-form on $\mathbb{C}^2$ which restricts to a
holomorphic one-form on the complement $\mathbb{C}^2\setminus F^{-1}(\Delta)$.
The Milnor fibration (\ref{milnor}) induces a representation
\begin{equation}\label{hmilnor}
\pi_1(\mathbb{C}\setminus\{\Delta\},*)\rightarrow
Aut(H_1(F^{-1}(t),\mathbb{C}))
\end{equation}
which implies the monodromy representation of the Abelian integral $I(t)$.

Let $V_t\subset H_1(F^{-1}(t),\mathbb{C})$ be a continuous family of complex
vector spaces obtained by a parallel transport. The space $V_t$ can be seen as
a point of the Grassmannian variety $Gr(d,H_1(F^{-1}(t),\mathbb{C}))$.
Therefore the representation (\ref{hmilnor}) induces an action of the
fundamental group $\pi_1(\mathbb{C}\setminus\{\Delta\},*) $ on
$Gr(d,H_1(F^{-1}(t),\mathbb{C}))$.
\begin{definition}
We say that a complex vector space $V_t\subset H_1(F^{-1}(t),\mathbb{C})$ of
dimension $d$ is virtually invariant, provided that its orbit in the
Grassmannian $Gr(d,H_1(F^{-1}(t),\mathbb{C}))$ under the action of
$\pi_1(\mathbb{C}\setminus\{\Delta\},*) $ is finite. A virtually invariant
space $V_t$ is said to be irreducible, if it does not contain non-trivial
proper virtually invariant subspaces.
\end{definition}
Let $\gamma(t)$ be a locally constant section of the homology bundle defined by
$F$. As intersection of virtually invariant vector spaces $V_t\subset
H_1(F^{-1}(t),\mathbb{C})$ containing $\gamma(t)$ is virtually invariant again,
then such an intersection is the minimal virtually invariant space containing
$\gamma(t)$. Clearly a virtually invariant minimal space containing $\gamma(t)$
need not be irreducible : it might contain a virtually invariant subspace
subspace not containing $ \gamma(t)$.

Consider the Abelian integral $I(t)= \int_{\gamma(t)}\omega$, where $\gamma(t)$
is a locally constant section of the homology bundle and $ \omega$ is a
meromorphic one-form as above. Denote by $V_t$ the minimal virtually invariant
vector space containing $\gamma(t)$.
\begin{theorem}
\label{th:algeq} If $V_t$ is irreducible, then either the Abelian integral
$I(t)$  vanishes identically, or its minimal annihilator is a linear
differential operator of degree $d=\dim V_t$.
\end{theorem}
\begin{proof}
Let $S_t$ be the complex vector space of germs of analytic functions in a
neighborhood of $t$, obtained from $I(t)$ by analytic continuation along a
closed path in $\mathbb{C}\setminus \Delta$. It suffice to check that $V_t$ is
isomorphic to $S_t$. Equivalently, for every locally constant section
$\delta(t)\in V_t$ we must show that $\int_{\delta(t)} \omega\not\equiv 0$.
Indeed, the vector space of all locally constant sections $\delta(t)$ with
$\int_{\delta(t)} \omega\equiv 0$ is an invariant subspace of $V_t$. As $V_t$
is supposed to be irreducible, then  this space is trivial.  Theorem
\ref{th:algeq} follows from Theorem \ref{th:main1}.
\end{proof}
The above theorem is easily generalized. For instance, the coefficients of the
minimal annihilator of $I$ are rational functions of $t$ if and only if the
minimal virtually invariant space $V_t$  containing $\gamma$ is monodromy
invariant, i.e.\ its orbit  in the Grassmannian consists of a single point.
Further, it might happen that $V_t$ is reducible. Let $V_t^0$ be a proper
virtually invariant subspace of $V_t$. If the factor space $V_t/V_t^0$ is
irreducible (does not contain proper virtually invariant subspaces), then
Theorem \ref{th:algeq} still holds true, but the minimal annihilator of $I(t)$
is of order equal to $\dim V_t - \dim V_t^0$. Multidimensional Abelian
integrals (along $k$-cycles) are studied in a similar way.

\section{Examples of Abelian integrals related to perturbation of the Lotka-Volterra system}
\label{sec:ex} Let $F$ be a real polynomial and $\omega=Pdx + Qdy$ a real
polynomial differential one-form in $\mathbb{R}^2$. Consider the perturbed real
foliation in $\mathbb{R}^2$ defined by

\begin{equation}\label{16th}
dF + \varepsilon \omega = 0 .
\end{equation}

The infinitesimal 16th Hilbert problem asks for the maximal number of limit
cycles of (\ref{16th}) when $\varepsilon \sim 0$ as a function of the degrees
of $F, P, Q$. Let $\gamma(t)\subset F^{-1}(t)$ be a continuous family of closed
orbits of (\ref{16th}). The zeros of the Abelian integral $I(t)=
\int_{\gamma(t)}\omega$ approximate limit cycles (at least far from the
atypical points of $F$) in the following sense. If $I(t_0)=0, I'(t_0)\neq 0$,
then a limit cycle of (\ref{16th}) tends to the oval $\gamma(t_0)$,  when
$\varepsilon$ tends to $t_0$. The question of explicit computing the number of
zeros  of Abelian integrals remains open (although a substantial progress was
recently achieved, see \cite{ilya02,bny} and the references therein).
Generically an Abelian integral satisfies a Picard-Fuchs differential equation
$$
I^{(d)} + a_1 I^{(d-1)}+\dots a_d I = 0, a_i\in \mathbb{R}(t)
$$
of order equal to the dimension of the homology group of the typical fiber
$F^{-1}(t)$. We are interested in the possibility of reducing the degree of
this equation, assuming that the coefficients of the equation are algebraic in
$t$, $a_i\in \mathbb{C}(t)^a$. Indeed, the zeros of the solutions of a second
order equation are easily studied (by the Rolle's theorem).

In this section we study Abelian integrals which appear in the perturbations of
foliations $dF=0$ with $F = x^p (y^2+x-1)^q$  and $F(x,y)=(xy)^p(x+y-1)$, where
$p,q$ are positive integers. The corresponding foliation $dF=0$  is a  special
Lotka-Volterra system.

\subsection{Toy example $F=x^p y^q$}

Consider first the fibration defined by the polynomial $F=x^p y^q$. We assume
that $p,q$ are \emph{relatively prime}. The base of the fibration is the
punctured plane $B=\bbC\setminus \{0\}$. Each fiber is a sphere with two points
removed. The homology bundle is one-dimensional with trivial monodromy
representation. We investigate the monodromy representation on the
\emph{relative homology} bundle. It will be a basic ingredient of the monodromy
investigation in more complicated cases.

Consider a set of marked points $B_t$ on the complex fibre $F^{-1}(t)$
\[
B_t=(F^{-1}(t)\cap \{x=L\})\cup (F^{-1}(t)\cap \{y=L\}),
\]
where $L$ is a real positive number. The relative homology $H_1(F^{-1}(t),B_t)$ is a free group with $p+q$ generators. A convenient model for the pair $(F^{-1}(t),B_t)$ consists of a cylinder with some strips attached; marked points are located at the ends of these strips.

Note that there exist a unique pair of positive integers $(m,n)$ satisfying the
following relation
\begin{equation}
\label{pqmn}
p\, m+q\, n =1,\qquad |m|<q, \quad |n|<p.
\end{equation}
Let $S\subset \bbC$ be the strip in complex plane around the real segment $[1,L]$. Let $C(r,R)\subset\bbC$ be ring (cylinder) which radii $r$ and $R$ satisfy relations $L^{-1}<r<1<R<L$. The model $M$ is a surface constructed with three charts $U_x$, $U_y$, $U_c$:
\begin{equation}
\begin{split}
U_x &= \{(x,\nu):\ x\in S,\ \nu\in \bbZ/q\},\\
U_y &=  \{(y,\mu):\ y\in S,\ \mu\in \bbZ/p\},\\
U_c &= \{u\in C(r,R) \}
\end{split}
\end{equation}
with the following transition functions (strips $U_x$ are attached to the external circle of radius $R$ and strips $U_y$ are attached to the internal boundary of $U_c$)
\begin{equation}
\label{trfn}
u(x,\nu)= x^{1/q} e^{2\pi i/q\, (-\nu m)},\qquad u(y,\mu)= y^{-1/p} e^{2\pi i/p\, (\mu n)}.
\end{equation}
The marked points are $\{x=L\}$ and $\{y=L\}$ at the end of strips. To construct a map $\psi_t$ we will use a bump function $\varphi\in C^\infty([0,1])$ which is 0 near $s=0$ and 1 near $s=1$. The map $\psi_t:M\rightarrow \bbC^2$ reads
\begin{equation}
\label{psidef}
\psi_t :
\left\{
\begin{aligned}
\psi_t(x,\nu) &= (x,\ t^{1/q} x^{-p/q} e^{2\pi i/q\,\nu})\\
\psi_t(y,\mu) &= (t^{1/p} y^{-q/p} e^{2\pi i/p\,\mu},\ y)\\
\psi_t(u) &= \big(u^q\, \exp(\tfrac {\log t}{p}\, \varphi(\tfrac{|u|-r}{R-r})),\  t^{1/q} u^{-p} \exp(-\tfrac{\log t}{q}\, \varphi(\tfrac{|u|-r}{R-r}) ) \big).\\
\end{aligned}\right.
\end{equation}
\begin{lemma}
\label{lem:toy}
The surface $M$ and the map $\psi_t$ provides a model of fiber for fibration defined by $F=x^p y^q$.
The monodromy transformation
$\mon :M\rightarrow M$ around $t_0=0$ reads
\begin{equation}
\label{toymon}
\mon:
\left\{
\begin{aligned}
\mon(x,\nu) &= (x,\nu+1)\\
\mon(y,\mu) &= (y,\mu+1)\\
\mon(u) &= u\, \exp\big(2\pi i(-\tfrac{m}{q}+\tfrac{1}{pq}\varphi (\tfrac{|u|-r}{R-r}))\big)
\end{aligned}\right.
\end{equation}
\end{lemma}

The surface $M$ and its monodromy transformation described in the above lemma are drawn on the figure \ref{fig:toymon3d}.

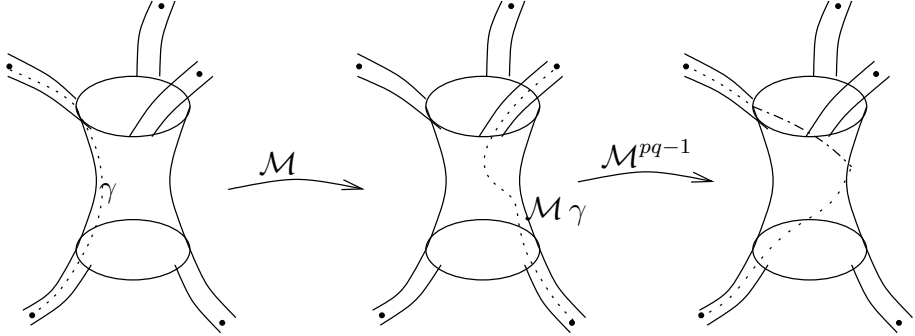
\begin{figure}[htpb]
\input{figs/toy3d.pspdftex}
\caption{Monodromy transformation of the model surface $M$ and a relative cycle $\gamma$.}
\label{fig:toymon3d}
\end{figure}

\begin{proof}
Complex level curves $F^{-1}(t)$ intersect line at infinity in two points: $[1:0:0]$ and $[0:1:0]$. The neighborhood of any of them is a punctured disc. Thus, there exists an isotopy of the level curve $F^{-1}(t)$ shrinking it to the region $\{|x|\leq R,\ |y|\leq R\}$ for sufficiently big $R$.

We will assume that $t$ is sufficiently close to $0$. The intersection of $F^{-1}(t)$ with the neighborhood $\{|x|\leq r,\ |y|\leq r\}$ of $(0,0)$ is a cylinder parametrized by the formula
\begin{equation}
\label{toyzerdisk}
u\mapsto (g^q\, u^q,\ g^{-p}\,u^{-p}t^{1/q}),
\end{equation}
where $g(t,u)$ is a function which will be fixed later.

The intersection of $F^{-1}(t)$ with set $\{|x|\leq R,\ |y|\leq R\}\setminus \{|x|\leq r,\ |y|\leq r\}$ decomposes into two connected components $V_x$ and $V_y$; one is located close to the $x$-plane and another to the $y$-plane respectively. The component $V_x$ is a graph of multi-valued ($q$-valued) function $y=t^{1/q}x^{-p/q}$ defined over the ring $\{r \leq|x|\leq R\}$. Marked points are images of point $x=L$ located on the real axis. We deform this domain by isotopy to the strip $S$ along real line -- see figure \ref{fig:toymod}.
\begin{figure}[htpb]
\input{figs/toydeform.pspdftex}
\caption{Deformation of domain to the strip $S$}
\label{fig:toymod}
\end{figure}
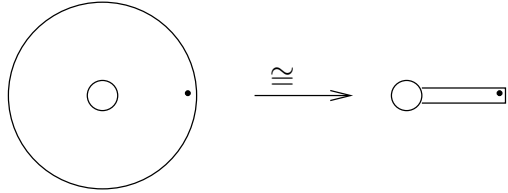
The values (leaves) of function $x^{-p/q}$ are numbered by $\nu\in\bbZ/q$. Thus, the domain $U_x$ and the map $\psi_t$ is defined as in lemma.

The model of $V_y$ is constructed in an analogous way.
To glue the above map together with parametrization \eqref{toyzerdisk} of disk around zero, we use the auxiliary function $g$. It must be equal $1$ near the internal circle of the ring $C(r,R)$ (i.e. $|u|=r$) and $t^{1/pq}$ near the exterior boundary ($|u|=R$). It is easy to check that $g=\exp(\tfrac{1}{pq}\log t\, \varphi(\tfrac{|u|-r}{R-r}))$ solves the problem.

The formula \eqref{toymon} for the monodromy around $t=0$ is a direct consequence of the formula \eqref{psidef}.\\
\end{proof}

A 2-dimensional version of figure \ref{fig:toymon3d} presenting the model surface $M$ is drawn below. It is obtained from figure \ref{fig:toymon3d} by cutting the cylinder along a vertical line. We will use these planar style of drawing models in subsequent, more complicated cases.

The segments that are unified are marked with arrows. Strips $U_x$ and $U_y$ are enumerated by integers $\tfrac{q}{2\pi}\arg u$ and $\tfrac{p}{2\pi}\arg u$ respectively; the argument $\arg u$ is calculated in point $u\in U_c$ which is glued with the point $1\in S$ according to relations \eqref{trfn}. Generators of the relative homology of $M$ are also marked.
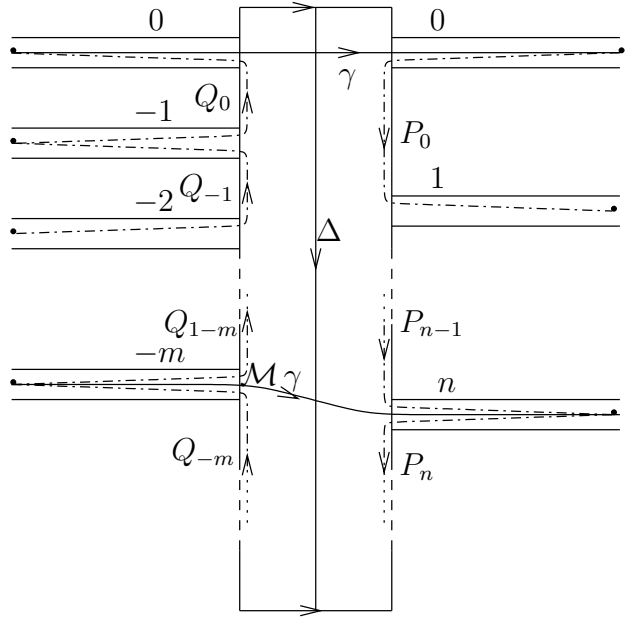
\begin{figure}[htpb]
\input{figs/toy2d.pspdftex}
\caption{The model surface $M$ with generators of the relative homology.}
\label{fig:toy2d}
\end{figure}

\begin{proposition}
\label{prop:toyhom}
The relative homology is $H_1(F^{-1}(t),B_t)$of the complex fiber $F^{-1}(t)$ has dimension $p+q$. It is generated by cycles
\[
\gamma,\Delta,Q_0,\ldots,Q_{q-1},P_0,\ldots,P_{p-1}
\]
with the relations:
\[
Q_0+\cdots+Q_{q-1}=-\Delta,\qquad P_0+\cdots+P_{p-1}=\Delta.
\]
The monodromy representation on the relative homology space reads:
\begin{equation}
\label{toymonhom}
\begin{split}
\mon Q_j=Q_{j-m},\qquad \mon P_k=P_{k+n},\qquad \mon \Delta = \Delta\\
\mon \gamma = \gamma+Q_0+\cdots+Q_{-m+1}+P_0+\cdots +P_{n-1}.
\end{split}
\end{equation}
\end{proposition}

The proposition is a direct consequence of lemma \ref{lem:toy}.

\vskip 1cm \subsection{The parabolic case} \label{sec:parab}

Consider the fibration given by a polynomial $F=x^p (y^2+x-1)^q$, where $p,q$ is a pair of positive, relatively prime integer numbers. Thus, they satisfy the relation \eqref{pqmn}
with a pair of integers $m,n$. They must be of opposite signs; we assume $m>0$ and so $n\leq 0$.

The base of locally trivial fibration in this case is a plane with two points removed $B=\bbC\setminus\{0,c\}$,
where $c=(\tfrac{p}{p+q})^p(\tfrac{-q}{p+q})^q$ corresponds to a center $(\tfrac{p}{p+q},0)$ of the Hamiltonian vector field $X_F$. The cycle $\gamma_t$ for $t\in (0,c)$ is an oval (compact component) of the real level curve $F^{-1}(t)$.

The model of complex fiber is presented on figure \ref{fig:parab}. It consists of two cylinders and $p+q$ strips glued together as shown on the figure. Cylinders are drown as rectangles, with horizontal edges unified. To simplify the combinatorial structure, there are opposite orientations on these two cylinders. Vertical, dotted lines mark another unification.
\begin{figure}[htpb]
\input{figs/parab.pspdftex}
\caption{Model of the level curve $F^{-1}(h)$ for $F=x^p(y^2 + x-1)^q$.}
\label{fig:parab}
\end{figure}
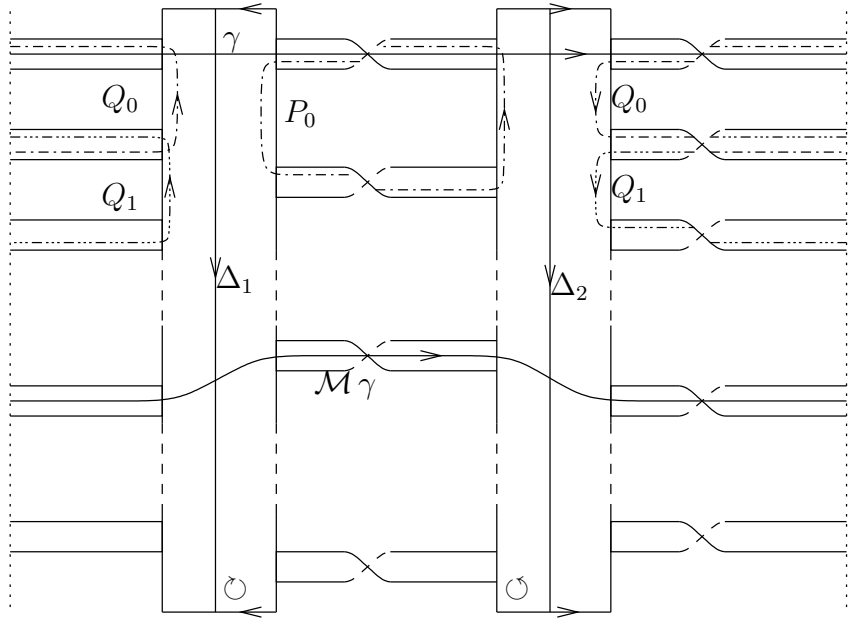

\begin{lemma}
\label{lem:parabmodel}
The surface shown on the figure \ref{fig:parab} provides a model $M$ for complex fiber $F^{-1}(t)$. The homology group $H_1(M)$ has dimension $p+q+1$ and is generated by cycles $\gamma,\Delta_1,\Delta_2,Q_0,\ldots,Q_{q-1},P_0,\ldots,Q_{q-1}$ with the following relations
\begin{equation}
\label{parel}
Q_0+\cdots+Q_{q-1} = \Delta_2-\Delta_1, \qquad, P_0+\cdots+P_{p-1}= \Delta_1-\Delta_2.
\end{equation}
Intersection indices of $\gamma$ and other generators of the homology group reads
\begin{equation}
\label{parints}
\begin{split}
\gamma\cdot Q_0=-1, \quad  \gamma\cdot Q_{q-1}=-1,\quad \gamma\cdot Q_j=0, \ \text{for}\ j=1,\ldots,q-2,\\
\gamma\cdot P_0=+1, \quad \gamma\cdot P_{p-1}=+1,\quad \gamma\cdot P_j=0, \ \text{for}\ j=1,\ldots,p-2,\\
\gamma\cdot\Delta_1 = +1,\qquad \gamma\cdot \Delta_2 = -1.
\end{split}
\end{equation}
The monodromy around zero  takes the form:
\begin{equation}
\label{parmon}
\begin{split}
\mon_0 Q_j = Q_{j+m},\qquad \mon_0 P_k = P_{k-n}, \qquad \mon_0 \Delta_j = \Delta_j \\
\mon_0 \gamma = \gamma+Q_0+\cdots+Q_{m-1}+P_0+\cdots +P_{-n+1}.
\end{split}
\end{equation}
\end{lemma}
\begin{proof}
The idea of proof is similar to the proof of lemma \ref{lem:toy}. We shrink the level curve $F^{-1}(t)$
by isotopy to the region $\{|x|\leq R,\ |y|\leq R\}$. We take the value of $t$ sufficiently close to 0. The intersections of $F^{-1}(t)$ with neighborhoods of saddle points $(0,1),\; (0,-1)$ are cylinders; we parametrize them by formulas similar to \eqref{toyzerdisk}. The remaining part of the fiber $F^{-1}(t)$ splits into two pieces: $V_l$ and $V_p$, located close to the line $x=0$ and close
to the parabola $y^2+x-1=0$ respectively. The part $V_l$ is the graph of $p$-valued function
$x=t^{1/p} (y^2+x-1)^{-q/p}$ defined over the disc of radius $R$ with small discs around points $y=\pm 1$ removed:
\[
U_l = \{y:\quad |y|\leq R,\ |y-1|\geq r,\ |y+1|\geq r\}.
\]
We deform the domain $U_l$ to the strip $S$ along the real segment -- see figure \ref{fig:pardeform}.
\begin{figure}[htpb]
\input{figs/parabdeform.pspdftex}
\caption{Deformation of domain $U_l$ to the strip $S$}
\label{fig:pardeform}
\end{figure}
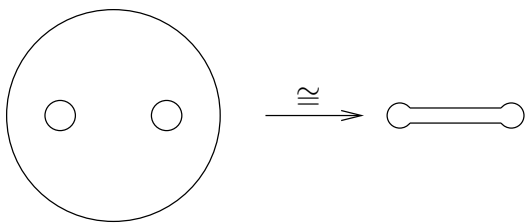
Leaves of function over the strip $S$ are numbered by $\mu\in\bbZ/p$. In analogous way we deform $V_p$ to the graph of the $q$-valued function defined over the strip along real segment of
the parabola $\{y^2+x-1=0\}$. Leaves of the function are numbered by $\nu\in\bbZ/q$.

We glue together both collections of strips with two cylinders in a way analogous to the toy example. Indeed,
in sufficiently small neighborhood of the point $(0,1)$ the pair of functions $(x,y^2+x-1)$ define a holomorphic
chart. In this chart the function $F$ takes the form as in the toy example. The same is true for the other
saddle $(0,-1)$. Both cylinders are glued by $p$ strips going along the line $x=0$ and $q$ strips going along
parabola $y^2+x-1$. The monodromy around zero permutes strips according to the rule
$\nu\mapsto \nu+1$, $\mu\mapsto \mu+1$, what is compatible with the formula \eqref{toymon} for monodromy in the
toy example case. Thus, the monodromy acts on both cylinders around saddles as in the toy example.

The surface shown on figure \ref{fig:parab} provides a model for complex fibre $F^{-1}(t)$. Formulas
\eqref{parmon} follow from respective formulas \eqref{toymonhom} in the toy example. The relations
\eqref{parel} and intersection indices \eqref{parints} one can read from the figure \ref{fig:parab}.\\
\end{proof}

\begin{corollary}
\[
\mon^{pq}_0 \gamma = \gamma + \Delta_2-\Delta_1.
\]
\end{corollary}
The  critical value $t=c$ corresponds to a Morse critical point of $F$. The
monodromy operator $\mon_c$ around $c$is therefore described by the usual
Picard-Lefschetz formula. Let $\gamma=\gamma(t)$ be the continuous family of
cycles, vanishing  $c$.
\begin{corollary}
\label{cor:parmoncen}
\begin{equation}
\label{parcenmon}
\begin{split}
\mon_c Q_0=Q_0 - \gamma, \qquad \mon_c Q_{q-1}=Q_{q-1}-\gamma,\\
\mon_c Q_j=Q_j, \ \text{for}\ j=1,\ldots,q-2,\\
\mon_c P_0=P_0+\gamma, \qquad \mon_c P_{p-1}=P_{p-1}+\gamma,\\
\mon_c P_j=P_j, \ \text{for}\ j=1,\ldots,p-2,\\
\mon_c\Delta_1 = \Delta_1+\gamma,\qquad \mon_c \Delta_2 = \Delta_2-\gamma.
\end{split}
\end{equation}
\end{corollary}

\begin{theorem}
\label{th:parab} The related Abelian integral $I=\int_\gamma \omega$ is either
identically zero, or it does not satisfy any differential equation with
algebraic coefficients of order $k< p+q-1$.
\end{theorem}
\begin{proof}
The proof is based on theorem \ref{th:algeq}. Let $H$ be a $k$-dimensional subspace of the (complex) homology space $H_1=H_1(F^{-1}(t),\bbC)$ and $\gamma\in H$. Assume that the monodromy orbit of $H$ in the Grassmannian $G_k(H_1)$ is finite. We show that the dimension of $H$ satisfies $\dim H \geq p+q$.

Let $\mon_0$ be the operator of monodromy around $t=0$ (i.e.\ along a loop winding once around $t=0$); let $\mon_c$ be a monodromy around the center $t=c$.
It follows formulas \eqref{parcenmon} that the $\mon_c - \idd$ is nilpotent operator and its image is one dimensional, generated by $\gamma$. The homology space $H_1$ splits into 2-dimensional $\mon_c$-invariant subspace $N$ and $(\dim H_1 -2)$-dimensional; the monodromy $\mon_c$ restricted to the latter one is the identity. The matrix of the restricted monodromy operator $\mon_c|_N$ in a basis $(\gamma,\delta)$ has the form
\[
[\mon_c|_N]_{(\gamma,\delta)}=\left(\begin{smallmatrix}1&1\\ 0&1\\ \end{smallmatrix}\right).
\]
Note that the subspace $N$ is not defined uniquely. It is spanned by $\gamma$ and any element $\delta\in H_1$ such that $\gamma\cdot \delta \neq 0$.

Consider the intersection $H\cap N$. The property that $H$ has a finite $\pi_1$ orbit (see theorem \ref{th:algeq}) implies that the $\mon_c$-orbit of $\mon_0^k H$, $k\in\bbZ$, is finite. Thus, the intersection $HN_k=(\mon_0^k H)\cap N$ has also finite $\mon_c$ orbit in $N$. The form of $\mon_c|_N$ implies that there are only 3 subspaces with a finite orbit:
\begin{equation}
\label{ninv}
HN_k=\{0\}, \qquad HN_k =\bbC\; \gamma, \qquad HN_k= N.
\end{equation}
Note that all these subspaces are $\mon_c$-invariant.
\begin{lemma}
\label{lem:ninvcnd}
Assume that the monodromy orbit of $H$ in $G_k(H_1)$ is finite. If $u\cdot \gamma \neq 0$ for an element
$u\in\mon^l_0 H$, $l\in\bbZ$, then $\gamma\in\mon^l_0 H$.
\end{lemma}
\begin{proof}
Take $\delta = u$ and consider 2-dimensional, $\mon_c$ invariant space $N$ spanned by $\gamma$ and $\delta$.
The $\mon_c$ orbit of the $\mon_0^l H$ space is finite, so the intersection $HN_l=N\cap\mon_0^l H$ has one of three
forms listed in \eqref{ninv}. Since $\delta \in HN_l$ then $HN_l=N$ and so $\gamma\in \mon_0^l H$.\\
\end{proof}

Consider the $\mon_0$-orbit of $\gamma\in H$. The lemma \ref{lem:ninvcnd} (for $l=-m$) and the property that the
intersection number is preserved by the monodromy implies the following condition
\begin{equation}
\gamma\cdot \mon_0^m \gamma \neq 0 \Rightarrow \mon_0^m \gamma\in H.
\end{equation}
Consider an element $\mon_0^{pq} \gamma = \gamma+(\Delta_2-\Delta_1)$. Since $\gamma\cdot (\Delta_2-\Delta_1)=-2$, then $(\Delta_2-\Delta_1)\in H$. Consider elements $\mon_0^l \gamma$ for $l=1,\ldots,pq$. The intersection index $| \gamma\cdot(\mon_0^l \gamma) |\leq N$. So, the cycles
\begin{equation}
\label{monkq}
\mon_0^{pq\, N+ q\, l}\gamma = \gamma + (N+lm)\,(\Delta_2-\Delta_1) + \sum_{j=0}^p a_j(l) P_j, \qquad l=1,\ldots,p
\end{equation}
have a nonzero intersection indices with $\gamma$. Since $p,q$ are relatively prime the space spanned sums $\sum_{j=0}^p a_j(l) P_j, \ l=1,\ldots,p$ coincide to the space generated by $(P_0+\cdots+P_{-n-1})$, $(P_{-n}+\cdots+P_{-2 n-1})$,\ldots; the latter one is the full space generated by $P_0,\ldots,P_{p-1}$. Both claims follow the fact that $p$ and $n$ are also relatively prime (see \eqref{pqmn}) and the following observation
\begin{lemma}
\label{linalg}
Let $V$ be a vector space of dimension $p$ and let $q$ be an integer. Assume that $p,q$ are relatively prime. Let $e_0,\ldots,e_{p-1}$ be a basis of $V$. Then the following sums
\begin{equation}
\label{qsums}
(e_0+\cdots+e_{q-1}),\ (e_{q}+\cdots+e_{2q-1}),\ \ldots (e_{(p-1)q}+\cdots+e_{pq-1})
\end{equation}
(all indices $\mod p$ assumed) generate the whole space $V$.
\end{lemma}
The proof of lemma is based on the following observations. Since $p,q$ are relatively prime, any sum of
length $q$ appears in a sequence \eqref{qsums}. The difference of two sums has the form
$e_j - e_{j+q}$, $j=0,\ldots,p-1$; they generate a hyperplane orthogonal to vector $e_0+\cdots+e_{p-1}$.
Since the scalar product $(e_0+\cdots+e_{p-1})\cdot(e_0+\cdots+e_{q-1})=q$, the space generated by
vectors \eqref{qsums} is a whole $V$.

Thus, it is proved that the subspace $H$ must contain the subspace generated by
$P_0,\ldots,P_{p-1}$. In a similar way we show that $H$ contains the subspace
generated by $Q_0,\ldots,Q_{q-1}$.

We have shown that the subspace of the homology group containing $\gamma$, with finite $\pi_1$-orbit
must be necessarily $\pi_1$-invariant hyperplane in the homology space $H_1$. It proves the theorem
for a \emph{generic} 1-form $\omega$ (when the zero subspace $Z_\omega=\{0\}$). To finish the proof
we show that either $\dim Z_\omega\leq 1$ or $Z_\omega = H$.

Consider an element
\[
H\cap Z_\omega\ni v = a\, \gamma + \sum_{j=0}^{p-1} \alpha_j P_j + \sum_{i=0}^{q-1} \beta_i Q_i
\]
and its images under the monodromy around $t=0$: $\mon_0^l v$. Since $Z_\omega$ is monodromy invariant, all elements $\mon_0^l v\in Z_\omega$. If the intersection index $\gamma\cdot\mon_0^{l_0}v\neq 0$, then monodromy around the center $t=c$ adds a multiple of $\gamma$, so $\gamma\in Z_\omega$. Then, it follows from the previous analysis that $Z_\omega=H$. Assume now that all intersection indices $\gamma\cdot \mon_0^l v=0$. The coefficient $a$ must then vanish, otherwise $\mon_0^{pq}$ adds the cycle $\Delta_2-\Delta_1$ which realizes intersection index $-2$. Consider monodromies $\mon_0^{q l}v$, $l=0,\ldots,p-1$. It preserves the expression $\sum_{i=0}^{q-1} \beta_i Q_i$. Vanishing of the intersection indices $\gamma\cdot \mon_0^{q l}v$ implies equations
\begin{equation}
\label{intsind}
\alpha_j + \alpha_{j+1} = \beta_0+\beta_{q-1}, \qquad j=0,1,\ldots,p-1.
\end{equation}
The solution of \eqref{intsind} depends on the parity of $p$. If $p$ is odd then all coefficients $\alpha_j$ are equal: $\alpha_j=\alpha=\tfrac12 (\beta_0+\beta_{q-1})$. If $p$ is even the solution of \eqref{intsind} reads:
\[
\alpha_{2 l} = \alpha_0,\quad  \alpha_{2l+1}=\alpha_1,\quad \alpha_0+\alpha_1=\beta_0+\beta_{q-1}.
\]
We repeat then the analogous analysis with iterations of $\mon_0^p$. We obtain the following form of $Z_\omega$
\begin{equation}
Z_\omega \cap H \subset
\begin{cases}
\{0\} & \text{for}\ p,q \text{ odd} \\
\mathrm{Span}(2 \sum_{j=1}^{p/2} P_{2j} + (\Delta_2-\Delta_1)) & \text{for $p$ even and $q$ odd}\\
\mathrm{Span}(2 \sum_{j=1}^{q/2} Q_{2j} - (\Delta_2-\Delta_1)) & \text{for $q$ even and $p$ odd}\\
\end{cases}
\end{equation}
Thus, $\dim Z_\omega \cap H \leq 1$ and so the theorem is proved.\\
\end{proof}
\begin{corollary}
We have actually shown that the Abelian integral does not satisfy any
differential equation with algebraic coefficients of order lower than the Fuchs
type equation with rational coefficients which follows the general theory.
\end{corollary}

\vskip 1cm \subsection{The special Lotka-Volterra case} \label{sec:hvg}

Consider a fibration given by a polynomial $F(x,y)=(xy)^p(x+y-1)$. It defines a locally trivial fibration defined over plane with two points removed $B=\bbC\setminus \{0,c\}$, where $c=F(\tfrac{p}{1+2 p},\tfrac{p}{1+2 p})$ corresponds to a center. The cycle $\gamma_t$ for $t\in (0,c)$ is an oval (compact component) of the real level curve $F^{-1}(t)$. Note, that the fibration has a Morse type singularity at $t=c$ and $\gamma_t$ is a vanishing cycle at the center.

Below we investigate the fibration and the monodromy representation on the sufficiently small neighborhood of $t=0$: $|t|<\ep_0$. The monodromy around center $t=c$ follows the Picard-Lefshetz formula. Thus, to determine the monodromy representation it is enough to investigate the monodromy around $t=0$ and intersection indices with the cycle $\gamma$.

The model of complex fiber is presented on figure \ref{fig:hvg} which should be understood as follows. Any rectangle represents a cylinder; unification of edges according to arrows assumed. Another unification is assumed on vertical, dotted lines.
\begin{figure}[htpb]
\input{figs/hvg.pspdftex}
\caption{Model of the level curve $F^{-1}(h)$ for $F=(xy)^p(x+y-1)$.}
\label{fig:hvg}
\end{figure}
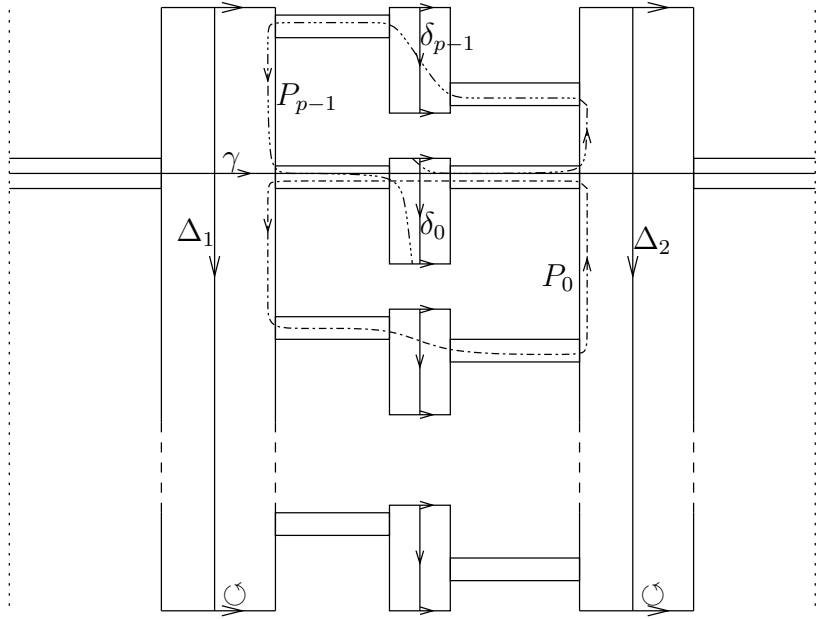

\begin{lemma}
\label{lem:hvgmodel}
The complex level curve $F^{-1}(t)$ is a surface of genus $p-1$ with $3$ points removed (intersection with the line at infinity). The surface shown on the figure \ref{fig:hvg} provides a model $M$ for $F^{-1}(t)$. The homology group $H_1(M)$ has dimension $2p+2$; it is generated by cycles $\gamma,\Delta_1,\Delta_2,P_0,\ldots,P_{p-1},\delta_0,\ldots,\delta_{p-1}$ with the following relation
\begin{equation*}
P_0+\cdots+P_{p-1} = \Delta_1-\Delta_2 + \delta_0.
\end{equation*}
Intersection indices of $\gamma$ with other generators of the homology group reads
\begin{equation}
\label{hvgints}
\begin{gathered}
\gamma\cdot P_{p-1}=-1, \qquad \gamma\cdot P_j=0, \quad \text{for}\ j=0,\ldots,p-2,\\
\gamma\cdot \delta_0=-1, \qquad \gamma\cdot \delta_j=0, \quad \text{for}\ j=1,\ldots,p-1,\\
\gamma\cdot\Delta_1 = -1,\qquad \gamma\cdot \Delta_2 = -1.
\end{gathered}
\end{equation}
The monodromy representation takes the form (monodromy around zero assumed)
\begin{equation}
\label{hvgmon}
\begin{gathered}
\mon \Delta_j = \Delta_j, \qquad \mon \delta_j=\delta_{j+1},\qquad \mon \gamma = \gamma+P_0,\\
\mon P_j = P_{j+1}, \quad \text{for}\ j=0,\ldots,p-2,\qquad\\
 \mon P_{p-1} = P_0+\delta_1-\delta_0.\\
\end{gathered}
\end{equation}
\end{lemma}
\begin{proof}
The proof is analogous to proofs of lemma \ref{lem:toy} and \ref{lem:parabmodel}. We modify the level curve $F^{-1}(t)$ by isotopy to the part contained in the compact region $|x|\leq R,\ |y|\leq R$. We consider points $t$ sufficiently close to $0$. We cut the level curve $F^{-1}(t)$ into pieces lying close to lines $x=0$, $y=0$, $x+y-1=0$ and close to saddles $(0,0),\, (1,0),\, (0,1)$. The analysis of pieces of level curve $F^{-1}(t)$ close to the saddles $(1,0)$ and $(0,1)$ and close to the line $x+y-1=0$ is completely analogous to the parabola case -- see proof of lemma \ref{lem:parabmodel}. The model of this part of the level curve consist of two cylinders joined by a single strip.

Now we consider the region which is in finite distance from the line $x+y-1=0$. The level curve $F^{-1}(t)$ outside the line $x+y-1=0$ splits into $p$ components defined by the equation:
\[
xy\, (x+y-1)^{1/p}=t^{1/p} \varepsilon_p^\nu,\qquad \nu=0,1,\ldots,p-1.
\]
Any of these components coincide to the toy example with $p=q=1$. Thus, it is isotopic to the cylinder with two strips attached. As $t$ winds around zero $t\mapsto e^{2\pi i} t$ we rotate components according to the rule $\nu\mapsto \nu+1 \mod p$. The $p$-th power of the monodromy (winding $p$ times around zero) $\mon^p$ corresponds to the usual monodromy in the toy example; it follows from the formula \eqref{toymonhom} (for $p=q=1$) that $\mon^p$ adds the generator $\delta_j$.

It proves that the combinatorial structure of model of the level curve $F^{-1}(t)$, defining how cylinders and strips are glued, must be as that shown on figure \ref{fig:hvg}.\\
\end{proof}

\begin{proposition}
\label{pr:hvg}
Let $H$ be the following 2-dimensional subspace
\begin{equation*}
H=\spn (\gamma,\; \Delta_1-\Delta_2+\delta_0)
\end{equation*}
of the (complex) homology space $H_1(F^{-1}(t))$. The orbit of $H$ under the monodromy representation $\pi_1\cdot H$ in Grassmannian $Gr_2(H_1)$ consists of $p$ elements and so is finite.
\end{proposition}
\begin{proof}
Denote by $\mon$ and $\mon_c$ the monodromy around $t=0$ and around the center critical value $t=c$. By lemma \ref{lem:hvgmodel}, we have
\begin{equation}
\label{hvgorbit}
\mon^k H = \begin{cases}
\spn \big( \gamma+P_0+\cdots+P_{k-1}, \; \Delta_1-\Delta_2+\delta_{k}\big) &\text{for}\  k=1,\ldots,p-1\\
\spn \big( \gamma+\Delta_1-\Delta_2+\delta_{0}, \; \Delta_1-\Delta_2+\delta_{0}\big)=H &\text{for}\  k=p.\\
\end{cases}
\end{equation}
The crucial observation is that subspaces $\mon^k H$ for $k\in\bbZ$ are $\mon_c$-invariant. Indeed, the subspace $H$ is $\mon_c$-invariant since $\gamma\in H$ is a vanishing cycle corresponding to the center critical value $t=c$. We calculate (using formulas \eqref{hvgints}) the intersection indices of $\gamma$ and generators of $\mon^k H$ for $k=1,\ldots,p-1$:
\[
\gamma\cdot \big( \gamma+(P_0+\cdots+P_{k-1})\big)=0,\qquad \gamma \cdot \big(\Delta_1-\Delta_2+\delta_{k})\big)=0.
\]
Thus, both generators are $\mon_c$-invariant. It proves that the orbit $\pi_1\cdot H$ in Grassmannian $Gr_2(H_1)$ consists of $p$ subspaces given in formula \eqref{hvgorbit}.\\
\end{proof}

\begin{corollary}
The proposition \ref{pr:hvg} provides a geometric explanation of phenomenon described in \cite{hvg}. According to the general theory given in theorem \ref{th:algeq} and calculations of monodromy given in lemma \ref{lem:hvgmodel} and proposition \ref{pr:hvg} the Abelian integral along cycle $\gamma$ satisfies a linear second order equation with algebraic coefficients in variable $t$.
\end{corollary}

\vskip 1cm
%

\end{document}

%% file: figs/toy3d.pspdftex
\begin{picture}(0,0)%
\includegraphics{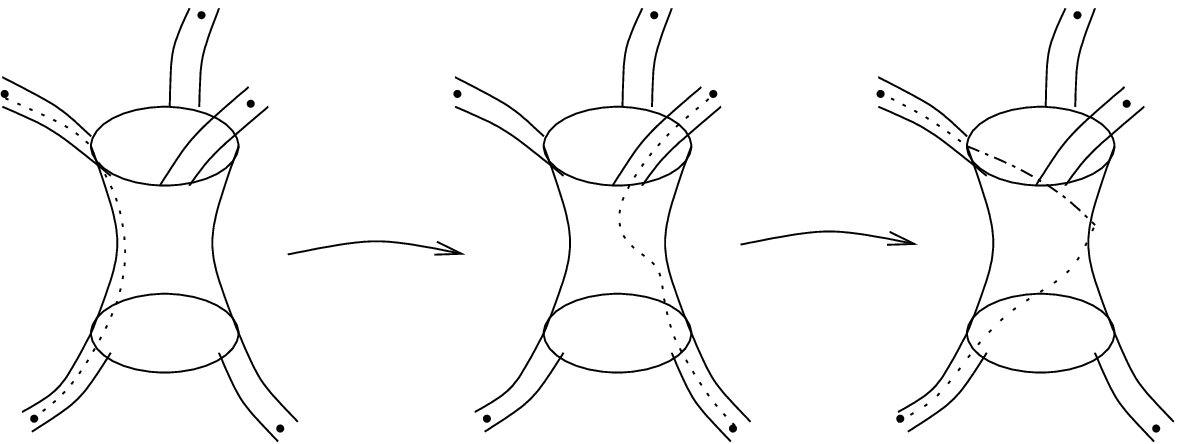}%
\end{picture}%
\setlength{\unitlength}{4144sp}%
\begingroup\makeatletter\ifx\SetFigFontNFSS\undefined%
\gdef\SetFigFontNFSS#1#2#3#4#5{%
  \reset@font\fontsize{#1}{#2pt}%
  \fontfamily{#3}\fontseries{#4}\fontshape{#5}%
  \selectfont}%
\fi\endgroup%
\begin{picture}(5379,2004)(259,-1603)
\put(1756,-646){\makebox(0,0)[lb]{\smash{{\SetFigFontNFSS{12}{14.4}{\rmdefault}{\mddefault}{\updefault}{\color[rgb]{0,0,0}$\mon$}%
}}}}
\put(3781,-601){\makebox(0,0)[lb]{\smash{{\SetFigFontNFSS{12}{14.4}{\rmdefault}{\mddefault}{\updefault}{\color[rgb]{0,0,0}$\mon^{pq-1}$}%
}}}}
\put(811,-781){\makebox(0,0)[lb]{\smash{{\SetFigFontNFSS{12}{14.4}{\rmdefault}{\mddefault}{\updefault}{\color[rgb]{0,0,0}$\gamma$}%
}}}}
\put(3331,-916){\makebox(0,0)[lb]{\smash{{\SetFigFontNFSS{12}{14.4}{\rmdefault}{\mddefault}{\updefault}{\color[rgb]{0,0,0}$\mon\gamma$}%
}}}}
\end{picture}%

%% file: figs/toydeform.pspdftex
\begin{picture}(0,0)%
\includegraphics{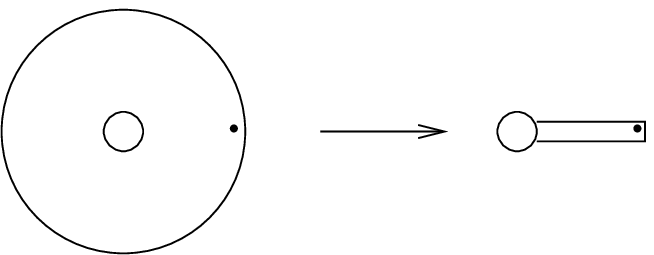}%
\end{picture}%
\setlength{\unitlength}{4144sp}%
\begingroup\makeatletter\ifx\SetFigFontNFSS\undefined%
\gdef\SetFigFontNFSS#1#2#3#4#5{%
  \reset@font\fontsize{#1}{#2pt}%
  \fontfamily{#3}\fontseries{#4}\fontshape{#5}%
  \selectfont}%
\fi\endgroup%
\begin{picture}(2962,1130)(-114,-176)
\put(1441,434){\makebox(0,0)[lb]{\smash{{\SetFigFontNFSS{12}{14.4}{\rmdefault}{\mddefault}{\updefault}{\color[rgb]{0,0,0}$\cong$}%
}}}}
\end{picture}%

%% file: figs/toy2d.pspdftex
\begin{picture}(0,0)%
\includegraphics{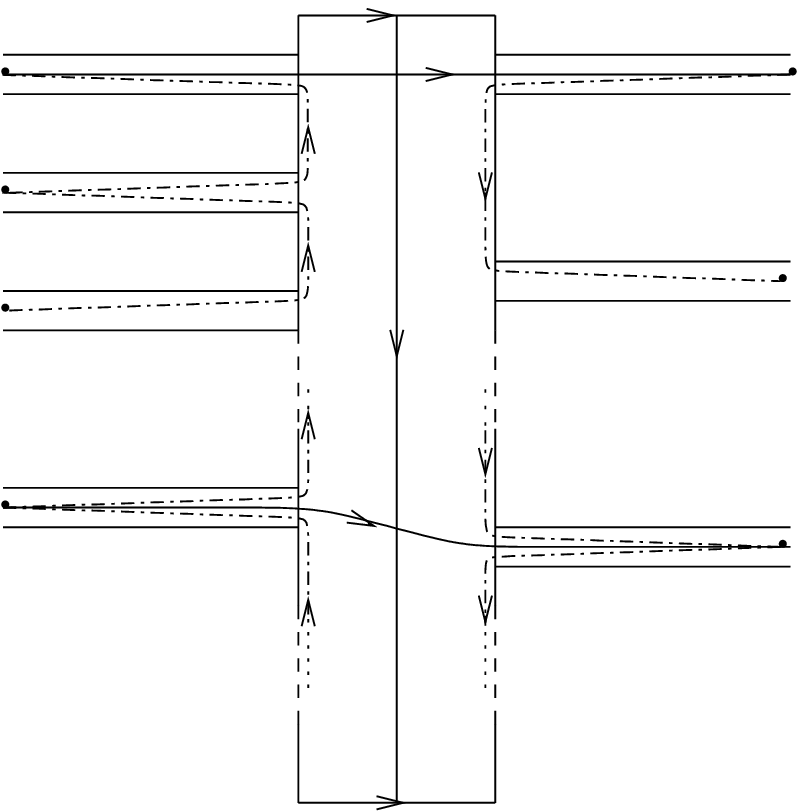}%
\end{picture}%
\setlength{\unitlength}{4144sp}%
\begingroup\makeatletter\ifx\SetFigFontNFSS\undefined%
\gdef\SetFigFontNFSS#1#2#3#4#5{%
  \reset@font\fontsize{#1}{#2pt}%
  \fontfamily{#3}\fontseries{#4}\fontshape{#5}%
  \selectfont}%
\fi\endgroup%
\begin{picture}(3646,3684)(887,-3703)
\put(2701,-1456){\makebox(0,0)[lb]{\smash{{\SetFigFontNFSS{12}{14.4}{\rmdefault}{\mddefault}{\updefault}{\color[rgb]{0,0,0}$\Delta$}%
}}}}
\put(2836,-511){\makebox(0,0)[lb]{\smash{{\SetFigFontNFSS{12}{14.4}{\rmdefault}{\mddefault}{\updefault}{\color[rgb]{0,0,0}$\gamma$}%
}}}}
\put(1621,-736){\makebox(0,0)[lb]{\smash{{\SetFigFontNFSS{12}{14.4}{\rmdefault}{\mddefault}{\updefault}{\color[rgb]{0,0,0}$-1$}%
}}}}
\put(1621,-1276){\makebox(0,0)[lb]{\smash{{\SetFigFontNFSS{12}{14.4}{\rmdefault}{\mddefault}{\updefault}{\color[rgb]{0,0,0}$-2$}%
}}}}
\put(1621,-2176){\makebox(0,0)[lb]{\smash{{\SetFigFontNFSS{12}{14.4}{\rmdefault}{\mddefault}{\updefault}{\color[rgb]{0,0,0}$-m$}%
}}}}
\put(3376,-196){\makebox(0,0)[lb]{\smash{{\SetFigFontNFSS{12}{14.4}{\rmdefault}{\mddefault}{\updefault}{\color[rgb]{0,0,0}$0$}%
}}}}
\put(3376,-1141){\makebox(0,0)[lb]{\smash{{\SetFigFontNFSS{12}{14.4}{\rmdefault}{\mddefault}{\updefault}{\color[rgb]{0,0,0}$1$}%
}}}}
\put(2251,-2311){\makebox(0,0)[lb]{\smash{{\SetFigFontNFSS{12}{14.4}{\rmdefault}{\mddefault}{\updefault}{\color[rgb]{0,0,0}$\mon \gamma$}%
}}}}
\put(3421,-2356){\makebox(0,0)[lb]{\smash{{\SetFigFontNFSS{12}{14.4}{\rmdefault}{\mddefault}{\updefault}{\color[rgb]{0,0,0}$n$}%
}}}}
\put(1711,-196){\makebox(0,0)[lb]{\smash{{\SetFigFontNFSS{12}{14.4}{\rmdefault}{\mddefault}{\updefault}{\color[rgb]{0,0,0}$0$}%
}}}}
\put(1981,-646){\makebox(0,0)[lb]{\smash{{\SetFigFontNFSS{12}{14.4}{\rmdefault}{\mddefault}{\updefault}{\color[rgb]{0,0,0}$Q_0$}%
}}}}
\put(1891,-1186){\makebox(0,0)[lb]{\smash{{\SetFigFontNFSS{12}{14.4}{\rmdefault}{\mddefault}{\updefault}{\color[rgb]{0,0,0}$Q_{-1}$}%
}}}}
\put(1846,-2761){\makebox(0,0)[lb]{\smash{{\SetFigFontNFSS{12}{14.4}{\rmdefault}{\mddefault}{\updefault}{\color[rgb]{0,0,0}$Q_{-m}$}%
}}}}
\put(1801,-1996){\makebox(0,0)[lb]{\smash{{\SetFigFontNFSS{12}{14.4}{\rmdefault}{\mddefault}{\updefault}{\color[rgb]{0,0,0}$Q_{1-m}$}%
}}}}
\put(3196,-871){\makebox(0,0)[lb]{\smash{{\SetFigFontNFSS{12}{14.4}{\rmdefault}{\mddefault}{\updefault}{\color[rgb]{0,0,0}$P_0$}%
}}}}
\put(3196,-1996){\makebox(0,0)[lb]{\smash{{\SetFigFontNFSS{12}{14.4}{\rmdefault}{\mddefault}{\updefault}{\color[rgb]{0,0,0}$P_{n-1}$}%
}}}}
\put(3196,-2851){\makebox(0,0)[lb]{\smash{{\SetFigFontNFSS{12}{14.4}{\rmdefault}{\mddefault}{\updefault}{\color[rgb]{0,0,0}$P_n$}%
}}}}
\end{picture}%

%% file: figs/parab.pspdftex
\begin{picture}(0,0)%
\includegraphics{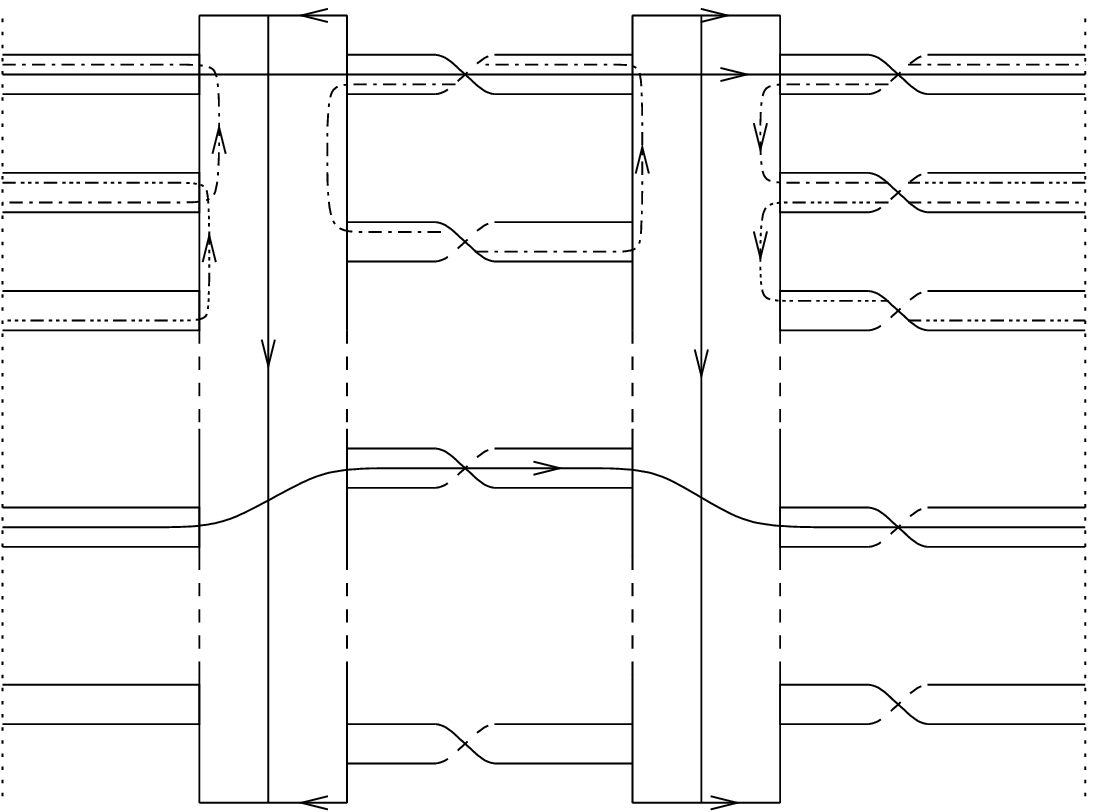}%
\end{picture}%
\setlength{\unitlength}{4144sp}%
\begingroup\makeatletter\ifx\SetFigFontNFSS\undefined%
\gdef\SetFigFontNFSS#1#2#3#4#5{%
  \reset@font\fontsize{#1}{#2pt}%
  \fontfamily{#3}\fontseries{#4}\fontshape{#5}%
  \selectfont}%
\fi\endgroup%
\begin{picture}(4974,3684)(1339,-3703)
\put(2566,-1726){\makebox(0,0)[lb]{\smash{{\SetFigFontNFSS{12}{14.4}{\rmdefault}{\mddefault}{\updefault}{\color[rgb]{0,0,0}$\Delta_1$}%
}}}}
\put(4546,-1771){\makebox(0,0)[lb]{\smash{{\SetFigFontNFSS{12}{14.4}{\rmdefault}{\mddefault}{\updefault}{\color[rgb]{0,0,0}$\Delta_2$}%
}}}}
\put(2971,-736){\makebox(0,0)[lb]{\smash{{\SetFigFontNFSS{12}{14.4}{\rmdefault}{\mddefault}{\updefault}{\color[rgb]{0,0,0}$P_0$}%
}}}}
\put(4906,-646){\makebox(0,0)[lb]{\smash{{\SetFigFontNFSS{12}{14.4}{\rmdefault}{\mddefault}{\updefault}{\color[rgb]{0,0,0}$Q_0$}%
}}}}
\put(1891,-646){\makebox(0,0)[lb]{\smash{{\SetFigFontNFSS{12}{14.4}{\rmdefault}{\mddefault}{\updefault}{\color[rgb]{0,0,0}$Q_0$}%
}}}}
\put(1891,-1231){\makebox(0,0)[lb]{\smash{{\SetFigFontNFSS{12}{14.4}{\rmdefault}{\mddefault}{\updefault}{\color[rgb]{0,0,0}$Q_1$}%
}}}}
\put(4906,-1186){\makebox(0,0)[lb]{\smash{{\SetFigFontNFSS{12}{14.4}{\rmdefault}{\mddefault}{\updefault}{\color[rgb]{0,0,0}$Q_1$}%
}}}}
\put(2611,-286){\makebox(0,0)[lb]{\smash{{\SetFigFontNFSS{12}{14.4}{\rmdefault}{\mddefault}{\updefault}{\color[rgb]{0,0,0}$\gamma$}%
}}}}
\put(3151,-2356){\makebox(0,0)[lb]{\smash{{\SetFigFontNFSS{12}{14.4}{\rmdefault}{\mddefault}{\updefault}{\color[rgb]{0,0,0}$\mon \gamma$}%
}}}}
\put(2611,-3571){\makebox(0,0)[lb]{\smash{{\SetFigFontNFSS{12}{14.4}{\rmdefault}{\mddefault}{\updefault}{\color[rgb]{0,0,0}$\circlearrowright$}%
}}}}
\put(4276,-3571){\makebox(0,0)[lb]{\smash{{\SetFigFontNFSS{12}{14.4}{\rmdefault}{\mddefault}{\updefault}{\color[rgb]{0,0,0}$\circlearrowleft$}%
}}}}
\end{picture}%

%% file: figs/parabdeform.pspdftex
\begin{picture}(0,0)%
\includegraphics{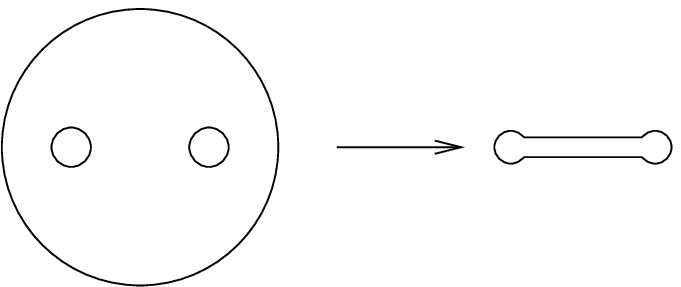}%
\end{picture}%
\setlength{\unitlength}{4144sp}%
\begingroup\makeatletter\ifx\SetFigFontNFSS\undefined%
\gdef\SetFigFontNFSS#1#2#3#4#5{%
  \reset@font\fontsize{#1}{#2pt}%
  \fontfamily{#3}\fontseries{#4}\fontshape{#5}%
  \selectfont}%
\fi\endgroup%
\begin{picture}(3079,1278)(-189,-250)
\put(1531,434){\makebox(0,0)[lb]{\smash{{\SetFigFontNFSS{12}{14.4}{\rmdefault}{\mddefault}{\updefault}{\color[rgb]{0,0,0}$\cong$}%
}}}}
\end{picture}%

%% file: figs/hvg.pspdftex
\begin{picture}(0,0)%
\includegraphics{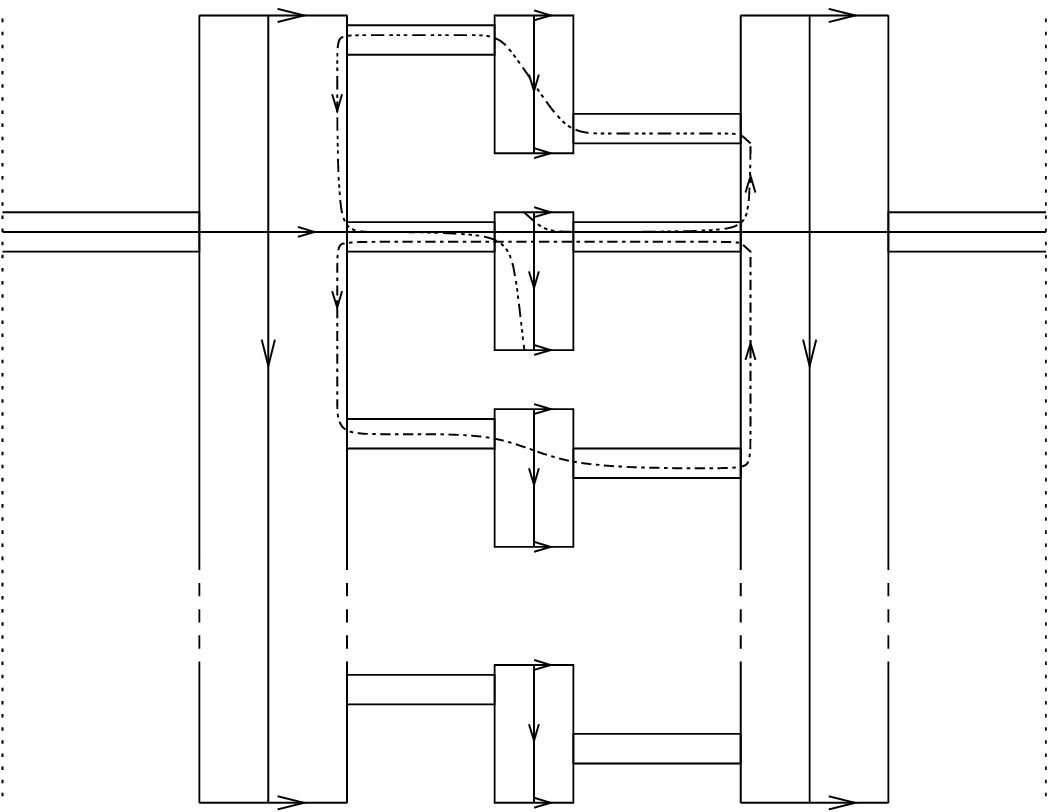}%
\end{picture}%
\setlength{\unitlength}{4144sp}%
\begingroup\makeatletter\ifx\SetFigFontNFSS\undefined%
\gdef\SetFigFontNFSS#1#2#3#4#5{%
  \reset@font\fontsize{#1}{#2pt}%
  \fontfamily{#3}\fontseries{#4}\fontshape{#5}%
  \selectfont}%
\fi\endgroup%
\begin{picture}(4794,3684)(1339,-3703)
\put(2611,-3616){\makebox(0,0)[lb]{\smash{{\SetFigFontNFSS{12}{14.4}{\rmdefault}{\mddefault}{\updefault}{\color[rgb]{0,0,0}$\circlearrowleft$}%
}}}}
\put(5086,-3616){\makebox(0,0)[lb]{\smash{{\SetFigFontNFSS{12}{14.4}{\rmdefault}{\mddefault}{\updefault}{\color[rgb]{0,0,0}$\circlearrowleft$}%
}}}}
\put(2611,-1006){\makebox(0,0)[lb]{\smash{{\SetFigFontNFSS{12}{14.4}{\rmdefault}{\mddefault}{\updefault}{\color[rgb]{0,0,0}$\gamma$}%
}}}}
\put(5041,-1501){\makebox(0,0)[lb]{\smash{{\SetFigFontNFSS{12}{14.4}{\rmdefault}{\mddefault}{\updefault}{\color[rgb]{0,0,0}$\Delta_2$}%
}}}}
\put(4501,-1726){\makebox(0,0)[lb]{\smash{{\SetFigFontNFSS{12}{14.4}{\rmdefault}{\mddefault}{\updefault}{\color[rgb]{0,0,0}$P_0$}%
}}}}
\put(3781,-1411){\makebox(0,0)[lb]{\smash{{\SetFigFontNFSS{12}{14.4}{\rmdefault}{\mddefault}{\updefault}{\color[rgb]{0,0,0}$\delta_0$}%
}}}}
\put(3781,-286){\makebox(0,0)[lb]{\smash{{\SetFigFontNFSS{12}{14.4}{\rmdefault}{\mddefault}{\updefault}{\color[rgb]{0,0,0}$\delta_{p-1}$}%
}}}}
\put(2926,-646){\makebox(0,0)[lb]{\smash{{\SetFigFontNFSS{12}{14.4}{\rmdefault}{\mddefault}{\updefault}{\color[rgb]{0,0,0}$P_{p-1}$}%
}}}}
\put(2341,-1456){\makebox(0,0)[lb]{\smash{{\SetFigFontNFSS{12}{14.4}{\rmdefault}{\mddefault}{\updefault}{\color[rgb]{0,0,0}$\Delta_1$}%
}}}}
\end{picture}%